\RequirePackage{fix-cm}
\documentclass[smallextended]{svjour3}       
\smartqed  
\usepackage{mathptmx,amsmath,xy,mathrsfs,verbatim,amssymb}
\numberwithin{theorem}{section}
\numberwithin{lemma}{section}
\numberwithin{proposition}{section}
\numberwithin{definition}{section}
\numberwithin{corollary}{section}

\newtheorem{Ass}{Assumptions}

\newcommand{\R}{\mathbb{R}}
\newcommand{\C}{\mathbb{C}}

\newcommand{\N}{\mathbb{N}}

\newcommand{\psim}{\textrm{psym}}
\renewcommand{\Re}{\mathfrak{Re}}
\newcommand{\Id}{\operatorname{Id}}

\newcommand{\csi}{\langle \xi \rangle}

\newcommand{\cl}{\mathrm{cl}}

\newcommand{\op}{\mathrm{Op}}

\usepackage{graphicx}

\begin{document}

\title{Weyl asymptotics of bisingular operators and
Dirichlet  divisor problem}

\author{Ubertino Battisti}

\institute{Ubertino Battisti, \at
              Universit\`a degli Studi di Torino,\\
               via Carlo Alberto 10, 10123 Torino \\
              Tel.: +39 011 6702877\\
              Fax: +39 011 6702878\\
              \email{ubertino.battisti@unito.it}                 
}

\date{Received: date / Accepted: date}

\maketitle

\begin{abstract}
We consider a class of pseudodifferential operators, with crossed vector valued symbols,
 defined on the product of two closed manifolds. We study the asymptotic expansion of the
counting function of positive selfadjoint operators in this class. Using a general Theorem of 
J. Aramaki, we can determine the first term of the asymptotic expansion of the counting function and, 
in a special case, we are able to find the second term. We give also some examples, emphasizing connections 
with problems of analytic number theory, in particular with Dirichlet divisor function.
\keywords{Weyl's law \and Bisingular operators \and Dirichlet divisor problem \and Spectral analysis}
 \subclass{35P20 \and 58J40  \and 47A10}
\end{abstract}

\section*{Introduction}
\label{intro}
In \cite{RO75} L. Rodino introduced bisingular operators: a class of pseudodifferential operators defined 
on the product of two closed manifolds $M_1 \times M_2$, related to the multiplicative property of  Atiyah-Singer 
index, see \cite{AS68}. A simple example of an operator in this class is the tensorial product
 $A_1 \otimes A_2$, where $A_1$, $A_2$ are pseudodifferential operators on the closed manifolds $M_1$, $M_2$. 
Another example, studied in \cite{RO75}, is the vector-tensor product $A_1 \boxtimes A_2$. In  \cite{NR06}, 
 in order to prove an index formula, F. Nicola and L. Rodino introduced classical, i.e. polyhomogeneous,
 bisingular operators and defined Wodzicki Residue for this class of operators. The two authors defined
 the residue, via holomorphic families, as in \cite{GL02,NI03}. For the index of bisingular operators 
see also the work of V. S. Pilidi \cite{PI73} and of R. V. Dudu{\v{c}}ava \cite{DU79I,DU79II}. In \cite{MR06},
R. Melrose and F. Rochon 
introduced pseudodifferential operators of product type, a class of operators 
close to bisingular operators. Bisingular operators are an example of operators with vector valued symbols; 
pseudodifferential operators of this type have
 been meticulously studied, see, for example, Fedosov, Schulze, Tarkhanov \cite{FST98} and the references therein.

The aim of this paper is to analyze the asymptotic behavior of  the counting function of selfadjoint elliptic
 positive bisingular operators. Similarly to the the case of SG-calculus \cite{BC10} (see e.g. \cite{ES97,SC87} 
for more detail on $SG$-calculus), we use techniques related to complex powers of operators, $\zeta$-function and 
Tauberian Theorems. This strategy, in the setting of closed manifolds, was first used by V. Guillemin \cite{GU85}
 in order to get the so called \emph{soft proof} of Weyl's formula. 

Here, as in the case of $SG$-calculus, it turns out that the $\zeta$-function can have poles of order two. Thus, 
using a refinement of Tauberian Theorem due to J. Aramaki \cite{AR88}, the asymptotic behavior of the 
 counting function is determined. The presence of a pole of order two of the $\zeta$-function implies that the counting
 functions can have asymptotic terms of order $\lambda^c \log \lambda$. Such a behavior appears in various setting:
 manifolds with conical singularities \cite{GL02}, $SG$-calculus on $\R^n$ \cite{NI03}, $SG$-calculus on manifolds with 
cylindrical ends \cite{MP02}. See also  Gramchev, 
Pilipovi{\'c}, Rodino, Wong  \cite{GPRW09,GPRW10} on the asymptotic expansion of the counting function in the case of twisted 
bi-Laplacian. Furthermore, in \cite{MO08},
 S. Moroianu  studied Weyl's law on manifolds
 with cusps, with an approach similar to the one used in this paper. 
In a special case, he  showed that the growth rate of the counting function is $\lambda^c \log \lambda$.

We remark that it is not surprising that the $\zeta$-function of a selfadjoint elliptic positive bisingular operator can have poles 
of order $2$. Indeed, let us consider two positive elliptic pseudodifferential operators $A, B$ defined 
on the closed manifolds $M_1, M_2$. From general theory of complex powers of pseudodifferential operators on closed 
manifolds \cite{SE67}, we know that the $\zeta$-function of an operator $P$ of this type is holomorphic for $\Re(z)<-\frac{n}{m}$
 ($n=\dim\; M$, $m$ order of $P$) and it can be extended as a meromorphic function to the whole of $\C$ with 
poles of order $1$. 
As we noticed at the beginning, the tensorial product $A \otimes B$ is a bisingular operator on $M_1 \times M_2$ and
it is clearly positive and selfadjoint. One can prove the following
\begin{equation}
\label{zetadoppia}
\zeta(A \otimes B, z)= \zeta(A,z) \zeta(B,z).
\end{equation}
If one defines the $\zeta$-function using the eigenvalues, equality \eqref{zetadoppia} becomes more transparent.
To this end, let $\{\lambda_j\}_{j \in \N }$ and $\{\mu_i\}_{i \in \N}$ be the eigenvalues of $A$ and $B$,
respectively. Then the eigenvalues of $A\otimes B$ turn out to be $\{\lambda_j \mu_i\}_{i,j \in \N^2}$. 
Therefore we have
\[
\begin{split}
\zeta(A,z)=&\sum_{j\in \N} \lambda_j^z, \quad \Re(z)<-\frac{n_1}{m_A}; \\
\zeta(B,z)=&\sum_{i\in \N}\mu_i^z, \quad \Re(z)<-\frac{n_2}{m_B};\\
\zeta(A \otimes B, z)= &\sum_{i,j \in \N^2} \lambda_j^z \mu_i^z= \zeta(A,z) \zeta(B,z), \quad \Re(z)<- \max\big\{ \frac{n_1}{m_A}, \frac{n_2}{m_B}\big\};
\end{split}
\]
where $n_1=\dim\; M_1$, $n_2= \dim \;M_2$ and $m_A,m_B$ are the orders of $A$ and $B$. 
Then the product structure of $\zeta(A \otimes B, z)$ implies that it can have poles of order two. 
Let us now focus on the special case $\frac{n_1}{m_A}=\frac{n_2}{m_B}=z_0$:
\begin{equation}
\label{zetaquadra}
\begin{split}
\zeta(A,z)=& \frac{C_A}{(z+z_0)}+ h_A(z), \quad \Re(z)<-z_0+\epsilon;\\
\zeta(B,z)=&\frac{C_B}{(z+z_0)}+ h_B(z), \quad \Re(z)< -z_0+\epsilon;\\
\zeta(A\otimes B,z)=& \frac{C_A C_B}{(z+z_0)^2}+ \frac{h_A(z)+ h_B(z)}{(z+z_0)}+ h_A(z) h_B(z), \quad \Re(z)< -z_0+\epsilon;
\end{split}
\end{equation}
where $C_A, C_B$ are constants that depend just on the principal symbol of $A,B$, while $h_A, h_B$ 
are holomorphic functions which depend on the whole symbol of $A, B$. From \eqref{zetaquadra}, it 
is clear that $\zeta(A \otimes B,z)$ has a pole of order two. Moreover, we observe that the coefficient
 of the pole of order one depends on the whole symbol of $A$ and $B$. Finally, applying
J. Aramaki's Theorem  \ref{aramaki}, from \eqref{zetaquadra} one obtains 
\begin{equation}
\label{aintintr}
N_{A\otimes B}(\lambda)\sim \frac{C_A C_B}{z_0} \lambda^{z_0} \log(\lambda)-
 \left(\frac{h_A(-z_0)- h_B(-z_0)}{z_0} +\frac{C_A C_B}{z_0^2}\right) \lambda^{z_0}+ O(\lambda^{z_0- \delta}),
\end{equation}
where $\delta>0$.
Simple examples of operators $A$ and $B$ for which \eqref{aintintr} holds are
$A= -\Delta_g+ 1$, $B=-\Delta_{g'}+1$, where $\Delta_g$, $\Delta_{g'}$ are the Laplace Beltrami operators
 associated to Riemanniann structures of $M_1$, $M_2$ respectively. We will extend \eqref{aintintr}
 to all positive bisingular elliptic operators, expressing the constants in the Weyl
 asymptotics in terms of the crossed vector-valued symbols. 

The paper is organized as follows. In Section \ref{sec:bis} we shortly recall basic properties of bisingular
 operators; we refer the reader to \cite{NR06,RO75} for more details. 
Section \ref{complexpowers} is devoted to the definition of complex powers of suitable bisingular operators;
 we introduce the $\zeta$-function in this setting and we study its meromorphic extension. 
The main result, concerning the asymptotics of the counting function of selfadjoint elliptic positive bisingular operators, 
is stated in section \ref{Weyl}. In section \ref{example}, we show the connection with  Dirichlet divisor problem, 
which we reconsider from the point of view of Spectral Theory.

\section{Bisingular operators}
\label{sec:bis}
	We start with the definitions of bisingular symbols and bisingular symbols with homogeneous principal symbol.
 In the following, $\Omega_i$ always denotes  a bounded open domain of $\R^{n_i}$.
	\begin{definition}
		We define $S^{m_1, m_2}(\Omega_1, \Omega_2)$ as the set of $C^{\infty}(\Omega_1 \times 				
	\Omega_2 \times \R^{n_1} \times \R^{n_2})$ functions such that, for all multiindex 			
				$\alpha_i, \beta_i$ and for all compact subset $K_i \subseteq \Omega_i$, $i=1,2$, 
there exists a 		
		positive constant $C_{\alpha_1, \alpha_2, \beta_1, \beta_2, K_1, K_2}$  so that 
		\[				
|\partial_{\xi_1}^{\alpha_1}\partial_{\xi_2}^{\alpha_2}\partial_{x_1}^{\beta_1}\partial_{x_2}^{\beta_2} 
a(x_1,x_2, \xi_1, \xi_2)|\leq C_{\alpha_1, \alpha_2, \beta_1, \beta_2, K_1, K_2}\langle \xi_1\rangle^{m_1-|\alpha_1|} 
\langle \xi_2\rangle^{m_2-|\alpha_2|},
		\]
	for all $x_i \in K_i$, $\xi_i \in \R^{n_i}$, $i=1,2$. As usual, $\csi= (1+ |\xi|^2)^{\frac{1}{2}}$.
	\end{definition}
\noindent $S^{-\infty, -\infty}(\Omega_1, \Omega_2)$ is  the set of smoothing symbols. Following \cite{RO75}, 
we introduce the subclass of bisingular operators with homogeneous principal symbol.

\begin{definition}
\label{simpri}
Let $a \in S^{m_1, m_2}(\Omega_1, \Omega_2)$; $a$ has a homogeneous principal symbol if
\begin{itemize}
\item[i)] there exists $a_{m_1, \cdot} (x_1, x_2, \xi_1, \xi_2) \in S^{m_1, m_2}(\Omega_1, \Omega_2)$ such that
\[
\begin{split}
& a(x_1, x_2, t \xi_1, \xi_2)= t^{m_1} a(x_1, x_2, \xi_1, \xi_2), \quad \forall x_1, x_2, \xi_2, \quad \forall |\xi_1|>1, t>0,\\
& a- \psi_1(\xi_1) a_{m_1, \cdot} \in S^{m_1-1, m_2}(\Omega_1, \Omega_2), \quad \psi_1 \;\mbox{cut-off function of the origin}.
\end{split}
\]
Moreover, $a_{m_1, \cdot}(x_1, x_2, \xi_1, D_2) \in L^{m_2}_{\cl}(\Omega_2)$, so, being a classical symbol on $\Omega_2$, 
it admits an asymptotic expansion w.r.t. the $\xi_2$ variable.
\item [ii)]
there exists $a_{\cdot, m_2} (x_1, x_2, \xi_1, \xi_2) \in S^{m_1, m_2}(\Omega_1, \Omega_2)$ such that
\[
\begin{split}
& a(x_1, x_2,  \xi_1, t\xi_2)= t^{m_2} a(x_1, x_2, \xi_1, \xi_2), \quad \forall x_1, x_2, \xi_1, \quad \forall |\xi_2|>1, t>0,\\
&a- \psi_2(\xi_2) a_{\cdot, m_2} \in S^{m_1, m_2-1}(\Omega_1, \Omega_2), \quad \psi_2\; \mbox{cut-off function of the origin}.
\end{split}
\]
Moreover, $a_{\cdot, m_2}(x_1, x_2, D_1, \xi_2) \in L^{m_1}_{\cl}(\Omega_1)$, so,  being a classical symbol on $\Omega_1$,
 it admits an asymptotic expansion w.r.t. the $\xi_1$ variable.
\item [iii)] The symbols $a_{m_1, \cdot}$ and $a_{\cdot, m_2}$
 have the same leading term, so there exists $a_{m_1, m_2}$ such that
\[
\begin{split}
&a_{m_1, \cdot}- \psi_2(\xi_2) a_{m_1, m_2} \in S^{m_1, m_2-1}(\Omega_1, \Omega_2), \\
&a_{\cdot, m_2}-\psi_1(\xi_1) a_{m_1, m_2} \in S^{m_1-1, m_2}(\Omega_1, \Omega_2),
\end{split}
\]
and 
\[
a- \psi_1 a_{m_1, \cdot}- \psi_2a_{\cdot, m_2}+ \psi_1 \psi_2 a_{m_1, m_2} \in S^{m_1-1, m_2-1}(\Omega_1, \Omega_2).
\]
\end{itemize}
The set of symbols with homogeneous principal symbol is denoted as $S^{m_1, m_2}_{\mathrm{pr}}(\Omega_1, \Omega_2)$. We will shortly
  write that the principal symbol of $a$ is $\{a_{m_1, \cdot}, a_{\cdot, m_2}\}$.
\end{definition}
We can observe a similarity, at least formal, between bisingular symbols with homogeneous principal symbol and $SG$-
classical symbols, see, e.g.. \cite{ES97,NI03}. 

We define bisingular operators via their left quantization. A linear operator 
$A: $ $C_c^{\infty}(\Omega_1$ $\times \Omega_2) \to C^{\infty}
(\Omega_1 \times \Omega_2)$ is a bisingular operator if it can be written in the form
\[
\begin{split}
A(u)(x_1, x_2)=&\op(a)(x_1, x_2)
\\=& \frac{1}{(2\pi)^{n_1+ n_2}} \int_{\R^{n_1}}\int_{\R^{n_2}} e^{i x_1 \cdot \xi_1+ i x_2 \cdot\xi_2}a(x_1, x_2, \xi_1, 
\xi_2)\hat{u}(\xi_1, \xi_2)d\xi_1 d\xi_2.
\end{split}
\]
If $a \in S^{m_1, m_2}(\Omega_1, \Omega_2)$ or $a \in S^{m_1, m_2}_{\mathrm{pr}}(\Omega_1, \Omega_2)$,
then we write $A \in L^{m_1, m_2}(\Omega_1, \Omega_2)$ and $A \in L_{\mathrm{pr}}^{m_1, m_2}(\Omega_1, \Omega_2)$  respectively.
The above definition can be extended to the product of closed manifolds; we refer to \cite{RO75} for the details of the construction of global 
operators and the corresponding calculus.

Definition \ref{simpri} implies that, for every operator $A \in L_{\mathrm{pr}}^{m_1, m_2}(\Omega_1, \Omega_2)$, we can define functions 
$\sigma^{m_1}, \sigma^{m_2}, \sigma^{m_1, m_2}$ such that
	\begin{equation}
	\label{isigma}
		\begin{split}
 		\sigma^{m_1}_1(A) : & T^*\Omega_1 \setminus\{0\} \to L^{m_2}_{\cl}(\Omega_2)\\
 								   & (x_1, \xi_1) \mapsto a_{m_1, \cdot}(x_1, x_2, \xi_1, D_2),	\\
		\sigma^{m_2}_2(A): & T^*\Omega_2\setminus \{0\} \to L_{\cl}^{m_1}(\Omega_1)\\
 								   & (x_2, \xi_2) \mapsto a_{\cdot,m_2}(x_1, x_2, D_1, \xi_2),\\
\sigma^{m_1, m_2}(A): & T^* \Omega_1 \setminus \{0\} \times T^*\Omega_2 \setminus \{0\} \to \C \\
 									 & (x_1, x_2, \xi_1, \xi_2) \mapsto a_{m_1, m_2}(x_1, x_2, \xi_1, \xi_2).
 	\end{split}
	\end{equation}		

	Moreover, denoting by $\sigma(P)(x, \xi)$ 
   the principal symbol of a preudodifferential operator $P$ on a closed manifold,
         the following  \emph{compatibility relation} holds
	\begin{equation}
          \begin{split}	
         \label{comppr}
	 \sigma (\sigma^{m_1}_1(A)(x_1,\xi_1))(x_2, \xi_2)=&
         \sigma(\sigma^{m_2}_2(A)(x_2,\xi_2))(x_1, \xi_1)\\
=&\sigma^{m_1, m_2}(A)(x_1, x_2, \xi_1, \xi_2)=a_{m_1, m_2} (x_1, x_2, \xi_1, \xi_2).
         \end{split}	
\end{equation}

\begin{remark}
If we consider the product of closed manifolds $M_1 \times M_2$, then the whole symbol  is a local object, in general.
 Nevertheless, similarly to the calculus on closed manifolds, it is possible to give an invariant meaning to the 
functions \eqref{isigma} as functions defined on the cotangent bundle, see \cite{RO75}. 
\end{remark}	
As in the case of the calculus on closed manifolds,
 it is possible to define adapted Sobolev spaces and then to prove some continuity results.
\begin{definition}
Let $M_1, M_2$ be two closed manifolds. The Sobolev space $H^{m_1, m_2}(M_1\times M_2)$ is defined by
\[
H^{m_1, m_2}(M_1 \times M_2)= \{u \in \mathscr{S}'(M_1 \times M_2) \mid 
\op(\langle \xi_1 \rangle^{m_1} \langle \xi_2\rangle^{m_2} )(u) \in L^2(M_1 \times M_2)\}.
\]
If $u \in H^{m_1, m_2}(M_1 \times M_2)$ then $\|u\|_{m_1, m_2}= \|\op(\langle \xi_1\rangle^{m_1} \langle \xi_2\rangle^{m_2} )(u)\|_2 $.
Using the formalism of tensor product, we can also write\footnote{For definition of $\widehat{\otimes}_\pi$ see \cite{TR67}.}
\[
H^{m_1, m_2}(M_1 \times M_2)= H^{m_1}(M_1) \widehat{\otimes}_\pi H^{m_2}(M_2).
\]
\end{definition}
Similarly to Sobolev spaces $H^s(M)$, we have  
\begin{itemize}
\item[i)] $H^{m_1, m_2}(M_1 \times M_2)\hookrightarrow H^{m_1', m_2'}(M_1 \times M_2)$ is a continuous immersion if $m_i\geq m_i' $, $i=1,2$.
\item[ii)]$H^{m_1, m_2}(M_1 \times M_2)\hookrightarrow H^{m_1', m_2'}(M_1 \times M_2)$ is a compact immersion if $m_i > m_i'$, $i=1,2$.
\end{itemize}
\begin{proposition}
A pseudodifferential operator $A \in L^{m_1, m_2}(M_1 \times M_2)$ can be extended to a continuous operator
\[
A: H^{s, t}(M_1 \times M_2) \to H^{s-m_1, t-m_2}(M_1 \times M_2).
\]
\end{proposition}
Furthermore, the norm of the operator can be estimated using the seminorms of the symbol. It is also possible to 
prove the following proposition: 
	\begin{proposition}
	\label{propnorm}
		Let $A \in L^{m_1, m_2}(M_1 \times M_2)$ be a bisingular operator; if $m_i \leq0$ ($i=1,2$), then 
there exists $N \in \N$ such that 
		$\|A\|_{0,0}\leq \sup |\sum_{i\leq N }p_i(a(x_1, x_2, \xi_1, \xi_2))|$, where $\{p_i(\cdot)\}_{i \in \N}$ are
 the seminorms of the Fr\'echet space $S^{m_1, m_2}(M_1, M_2)$.
			\end{proposition}

An operator $A \in L^{m_1, m_2}(M_1 \times M_2)$ 
is elliptic if  $a_{m_1, \cdot}, a_{\cdot, m_2}, a_{m_1, m_2}$, the three components of its principal symbol,
  are invertible in their 
domain of definition. Explicitly: 
\begin{definition}
	\label{defell}
	Let $A \in L_{\mathrm{pr}}^{m_1, m_2}(M_1 \times M_2)$; $A$  is elliptic if
	\begin{itemize}
		\item[i)] $\sigma^{m_1, m_2}(A)(v_1, v_2)\not=0$ for all $(v_1, v_2) \in T^*M_1 \setminus\{0\} \times T^*M_2 \setminus \{0\}$;
		\item[ii)]$\sigma^{m_1}_1(A)(v_1) \in L^{m_2}_\cl(M_2)$ is invertible for all $v_1 \in T^*M_1 \setminus\{0\}$;
		\item[iii)] $\sigma^{m_2}_2(A)(v_2)\in L^{m_1}_\cl(M_1)$ is invertible for all $v_2 \in T^*M_2 \setminus\{0\}$; 
	\end{itemize}
	where $\sigma^{m_1, m_2}(A), \sigma^{m_1}_1(A)$, $\sigma^{m_2}_2(A)$ are as in \eqref{isigma}.
\end{definition}
In \cite{RO75}, it is proved that, if $A$ satisfies  Definition \ref{defell}, then $A$ is a Fredholm operator. This property is a corollary of the following theorem:
\begin{theorem}
\label{parbis}
	Let $A \in L_{\mathrm{pr}}^{m_1, m_2}(M_1 \times M_2)$ be elliptic; then there exists an operator
	$B \in L^{-m_1, -m_2}_{\mathrm{pr}}(M_1 \times M_2)$ such that
	\[
	\begin{split}
	AB= \Id + K_1,\\
	BA= \Id+ K_2,
	\end{split}
	\]
	where $\Id$ is the identity map  and $K_1, K_2$ are compact operators. 
Moreover, the symbol of $B$ is $b=\{\sigma^{m_1}_1(A)^{-1} , \sigma^{m_2}_2(A)^{-1}\}$.
\end{theorem}
The proof of Theorem \ref{parbis} is an easy consequence of the global version of the following lemma:
\begin{lemma}
\label{lemcom}
Let $A \in L^{m_1, m_2}(\Omega_1 \times \Omega_2)$ and $B \in L^{m_1', m_2'}(\Omega_1 \times \Omega_2)$, then
\[
\{ (a\circ b)_{m_1+m_1', \cdot}, (a \circ b)_{\cdot, m_2+ m_2'}\}= 
\{ a_{m_1, \cdot}\circ_{\xi_2} b_{m_1', \cdot},  a_{ \cdot, m_2}\circ_{\xi_1}b_{\cdot, m_2'}\}
\]
where 
\[
\begin{split}
&(a\circ_{\xi_1} b)(x_1, x_2, D_1, \xi_2) (u)= a(x_1, x_2, D_1, \xi_2)  \circ b(x_1, x_2, D_1, \xi_2)(u) \quad \forall 
u \in C^\infty_c (\Omega_1),\\
&(a\circ_{\xi_2} b)(x_1, x_2, \xi_1, D_2)(v) = a(x_1, x_2, \xi_1, D_2) \circ b(x_1, x_2, \xi_1, D_2) (v) \quad \forall 
v \in C^\infty_c (\Omega_2).
\end{split}
\]
In first row the composition is in the space $L^{\infty}(\Omega_1)$ of pseudodifferential
 operators on $\Omega_1$, in second row, it is in the space $L^{\infty}(\Omega_2)$.
\end{lemma}


\section{Complex powers of bisingular operators}
\label{complexpowers}
In this section we define complex powers of a subclass of elliptic bisingular operators.  The first step is to give a suitable definition
$\Lambda$-elliptic operators w.r.t. a sector of the complex plane $\Lambda$.
\begin{definition}
\label{parell}
Let $\Lambda$ be a sector of $\C$; we say that $a \in S^{m_1, m_2}_{\mathrm{pr}}(M_1, M_2)$ is $\Lambda$-elliptic w.r.t. $\Lambda$ if there exists
 a positive constant $R$ such that
\begin{itemize}
		\item[i)] 
		\[
		\big(\sigma^{m_1, m_2}(A)(v_1, v_2) -\lambda\big)^{-1} \in S^{-m_1,-m_2}(M_1, M_2), 
		\]
		for all $|v_i|>R$, $i=1,2$, and for all $\lambda \in \Lambda$.
		\item[ii)]
		\[
		\sigma^{m_1}_1(A)(v_1)-\lambda\; \Id_{M_2} \in L^{m_2}_\cl(M_2),
		\] 
		is invertible for all $|v_1|>R$  and for all $\lambda \in \Lambda$.
		\item[iii)] 
		\[
		\sigma^{m_2}_2(A)(v_2)- \lambda \;\Id_{M_1} \in L^{m_1}_\cl(M_1),
		\] 
		is invertible for all $|v_2|>R$  and for all $\lambda \in \Lambda$. 
\end{itemize}
\end{definition}
In the following, 
in order to define the complex power of $A$, we assume that $\Lambda$ is a sector
 of the complex plane with vertex at the origin, 
that is
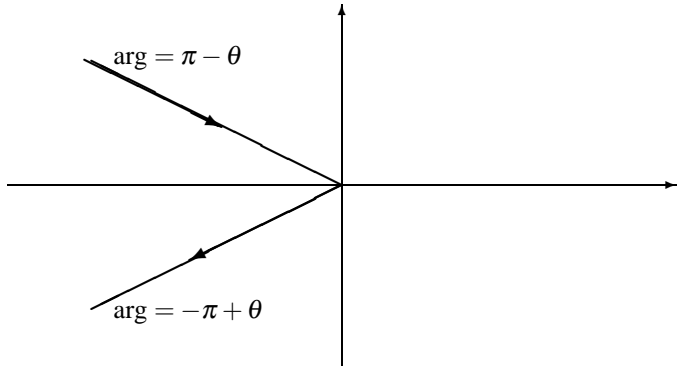
\[
\Lambda=\{z \in \C \mid \arg(z) \in [\pi-\theta, -\pi+\theta]\}.
\]

\setlength{\unitlength}{0.4cm}
\begin{picture}(22,12)
\thicklines

\put(11, 6){\line(-2,1){8.2284}}
\put(2.5424, 10.1653){\vector(2,-1){4.5}}
\put(11,6){\vector(-2,-1){5}}
\put(11, 6){\line(-2,-1){8.2284}}

\thinlines
\put(11,6){\vector(0,1){6}}
\put(11,6){\line(0,-1){6}}
\put(11,6){\vector(1,0){11}}
\put(11,6){\line(-1,0){11}}

\put(3.5,10){{arg} $= \pi - \theta$}
\put(3.5,1.6){{arg} $= -\pi + \theta$}
\end{picture}

\begin{lemma}
\label{eqell}
Let $a \in S^{m_1, m_2}(\Omega_1, \Omega_2)$ be $\Lambda$-elliptic. For all $K_i \subseteq \Omega_i$, $i=1,2$, there exist $c_0>1$ and a set
\begin{equation}
\label{lemmatec}
\Omega_{\xi_1, \xi_2}:=\{z \in \C \setminus \Lambda \mid \frac{1}{c_0} \langle \xi_1 \rangle^{m_1} \langle \xi_2 \rangle^{m_2} <|z|< c_0 \langle \xi_1 \rangle^{m_1} \langle \xi_2 \rangle^{m_2}\}
\end{equation}
such that
\[
\begin{split}
\mathrm{spec}(a(x_1,x_2,\xi_1,\xi_2))= \{\lambda\in \C \mid a(x_1,x_2, \xi_1, \xi_2)-\lambda=0 \}&\subseteq \Omega_{\xi_1, \xi_2}, \\
&\forall x_i \in \Omega_i, \xi_i \in \R^{n_i};
\end{split}
\]
moreover,
\[
\begin{split}
& |\big(\lambda-a_{m_1, m_2}(x_1, x_2, \xi_1, \xi_2)\big)^{-1}|\leq C (|\lambda|+ \langle \xi_1 \rangle^{m_1} \langle \xi_2 \rangle^{m_2})^{-1}, \\
& |\big( a_{m_1, \cdot} -\lambda \Id_{\Omega_1} \big)^{-1} |\leq C (|\lambda|+ \langle \xi_1 \rangle^{m_1} \langle \xi_2 \rangle^{m_2})^{-1},\\
& |\big( a_{\cdot, m_2}-\lambda \Id_{\Omega_2}\big)^{-1} |\leq C (|\lambda|+ \langle \xi_1 \rangle^{m_1} \langle \xi_2 \rangle^{m_2})^{-1},\\
& \forall x_i \in K_i, \xi_i \in \R^{n_i}, \lambda \in \C \setminus \Omega_{\xi_1, \xi_2}, i=1,2,
\end{split}
\]
where $\big(a_{m_1, \cdot}-\lambda \Id_{\Omega_{1}}\big)^{-1}$ stands for the symbol of the operator $(a_{m_1, \cdot}(x_1, x_2, \xi_1, D_2)-
 \lambda \Id_{\Omega_1})^{-1}$, and similarly for $\big(a_{\cdot, m_2}- \lambda \Id_{\Omega_2}\big)^{-1}$.
\end{lemma}
The proof of Lemma \ref{eqell} is essentially the same of the one of Lemma 3.5 in \cite{MSS06}.

Next, we prove that, if $A$ $\Lambda$-elliptic, then we can define a parametrix of $(A-\lambda \, \Id)$. 
Actually, we  prove that, for $|\lambda|$ large enough, the resolvent $(A-\lambda\, \Id)^{-1}$ exists.
Restricting ourselves to differential operators, we could follow formally the idea of Shubin (\cite{SH87}, ch. II) of parameter depending operators.
 For general pseudodifferential operators, it is well know that this idea does not work, see \cite{GS95}. 

\begin{theorem}
\label{thmpara}
Let $A \in L_{\mathrm{pr}}^{m_1, m_2}(M_1 \times M_2)$ be $\Lambda$-elliptic. Then there exists $R \in \R^+$, such that the resolvent
 $(A-\lambda \Id)^{-1}$ exists for $\lambda \in \Lambda_R=\{\lambda \in \Lambda \mid |\lambda|\geq R\}$. Moreover,
\[
\|(A-\lambda \Id)^{-1}\| = O(|\lambda|^{-1}), \quad \lambda \in \Lambda_R.
\]
\end{theorem}
\begin{proof}
First, we look for an inverse of $(A-\lambda \, \Id)$ modulo compact operators, that is an operator $B(\lambda)$ such that:
\begin{equation}
\label{ris}
\begin{split}
(A-\lambda)\circ B(\lambda)&=\Id+ R_1(\lambda), \quad \lambda R_1(\lambda) \in L^{-1, -1}(M_1 \times M_2),\\
B(\lambda)\circ (A-\lambda)&= \Id+R_2(\lambda), \quad \lambda R_2(\lambda) \in L^{-1, -1}(M_1 \times M_2),
\end{split}
\end{equation}
uniformly w.r.t. $\lambda \in \Lambda$. 
In order to find such an operator, we make  the principal symbol explicit:
\[
a-\lambda= \psim(a)-\lambda + c, \quad c \in S^{m_1-1, m_2-1}(M_1, M_2),
\]
where $\psim(a)= \psi_1 a_{m_1, \cdot}+ \psi_2 a_{\cdot, m_2}-\psi_1 \psi_2a_{m_1, m_2} $. 
As we have noticed in Theorem \ref{parbis}, we can write the symbol of the inverse (modulo compact operators) 
of an elliptic operator.
 In this case we need to be more careful because of the parameter $\lambda$. 
Following the same construction as in Theorem \ref{parbis}, we obtain
\begin{equation}
\label{parametrix}
b(\lambda)=\{ \big((\sigma^{m_1}_1(A)-\lambda \,\Id_{M_2} )^{-1}, (\sigma^{m_2}_2(A)-\lambda \,\Id_{M_1})^{-1}\}.
\end{equation}
The above definition \eqref{parametrix} is consistent in view of the $\Lambda$-ellipticity and of the following relation
\[
\begin{split}
\sigma \big((\sigma^{m_1}_1(A)-\lambda \,\Id_{M_2} )^{-1}(x_1, \xi_1)\big) (x_2, \xi_2)&=(a_{m_1, m_2}-\lambda)^{-1}(x_1, x_2, \xi_1, \xi_2),\\
\sigma \big((\sigma^{m_2}_2(A)-\lambda \,\Id_{M_1} )^{-1}(x_2, \xi_2) \big)(x_1, \xi_1)&=(a_{m_1, m_2}-\lambda)^{-1}(x_1, x_2, \xi_1, \xi_2).
\end{split}
\]
Using the rules of the calculus and Lemma \ref{eqell}, we can check that $B(\lambda)$ satisfies conditions
 \eqref{ris}. By  parameter ellipticity, we get that $R_1(\lambda)$ and $R_2(\lambda)$ are compact operators for $\lambda \in \Lambda$,
namely
\begin{equation}
\label{par}
\begin{split}
(A -\lambda \,\Id) \circ B(\lambda)&= \Id + R_1(\lambda), \\
(A -\lambda \,\Id) \circ B(\lambda)&= \Id + R_2(\lambda),
\end{split}
\end{equation}
$\lambda R_1(\lambda), \lambda R_2(\lambda) \in S^{-1, -1}(M_1 \times M_2) $ uniformly w.r.t.  $\lambda \in \Lambda$.
So $B(\lambda)$ is a parametrix and its symbol $b(\lambda)$ has the following form
\[
\begin{split}
&b(\lambda)=-(a_{m_1, m_2}(x_1, x_2, \xi_1, \xi_2)-\lambda)^{-1}\psi_1(\xi_2) \psi_2(\xi_1)\\
&+ (a_{m_1,\cdot}-\lambda\, \Id_{M_2})^{-1}(x_1, x_2, \xi_1, \xi_2) \psi_1(\xi_1) \\
&+(a_{\cdot, m_2} -\lambda\, \Id_{M_1})^{-1} (x_1, x_2, \xi_1, \xi_2) \psi_2(\xi_2),
\end{split}
\]
where $(a_{m_1,\cdot}-\lambda\, \Id_{M_2})^{-1} (x_1, x_2, \xi_1, \xi_2)$ is the value of the symbol
of the operator $(a_{m_1,\cdot}(x_1, x_2, \xi_1, D_2) -\lambda\, \Id_{M_2})^{-1}$ at $(x_2, \xi_2)$, and similarly for 
$(a_{\cdot, m_2} -\lambda\, \Id_{M_1})^{-1}$.
Furthermore, denoting by $r_1(\lambda)$ the symbol of $R_1(\lambda)$, we easily obtain
\begin{equation}
\label{resto}
r_1(\lambda)= (a-\psim(a))\circ b(\lambda)+ (\psim(a) \circ b(\lambda))-1,
\end{equation}
hence $r_1(\lambda) \in S^{-1, -1}(M_1, M_2)$ is the asymptotic sum of terms of the type
\[
\partial_{\xi_1}^{\alpha_1} \partial_{\xi_2}^{\alpha_2} g D_{x_1}^{\alpha_1}D_{x_2}^{\alpha_2} 
b(\lambda)\quad g \in S^{m_1, m_2} (M_1, M_2).
\]
Clearly $(a_{m_1, m_2}(x_1, x_2, \xi_1, \xi_2)-\lambda)^{-1} = O(|\lambda|^{-1})$. By the theory of pseudodifferential
 operators on closed manifolds, the same property
holds for the symbols of the operators $ (a_{m_1,\cdot}(x_1, x_2, \xi_1, D_2)-
\lambda \, \Id_{M_2})^{-1}$ and $(a_{\cdot, m_2}(x_1, x_2, D_1, \xi_2)-\lambda\,  \Id_{M_1})^{-1}$ 
and their derivatives. Thus $r_1(\lambda) = O(|\lambda|^{-1})$, as a consequence of the calculus. By Proposition \ref{propnorm}, 
this implies  $\|R_1\|_{L^2} = O(|\lambda|^{-1})$, and the same is true for
the operator $R_2$. So we can choose $\lambda$ large enough such that
 $R_1, R_2$ have norm less than $1$. In this way, using Neumann series, we prove that $(A-\lambda \,\Id)$ is
one to one and onto, therefore
 invertible, by the Open Map Theorem. Again, by Neumann series, we obtain $\tilde{B}(\lambda)$
 such that \eqref{par} is fulfilled with
 $\tilde{R}_1, \tilde{R}_2$ smoothing and still with norm $O(\lambda^{-1})$. Now notice that
 $\lambda \big[ B(\lambda)- \tilde{B}(\lambda)\big]$ $
 \in S^{-m_1-1, -m_2-1}$ for all $\lambda \in \Lambda$.
Furthermore, if we multiply both equations in \eqref{ris} by $(A-\lambda \Id)^{-1}$ we obtain
\[
(A-\lambda\, \Id)^{-1}=\tilde{B}(\lambda)+ \tilde{B}(\lambda)R_1(\lambda)+R_2(\lambda)(\lambda -A)^{-1}R_1(\lambda).
\]   
Hence $\|(A-\lambda \Id)^{-1}\| = O(|\lambda|^{-1})$ and  $\lambda^2 \big[(A-\lambda)^{-1} - 
\tilde{B}(\lambda)\big]$ 
is a smoothing operator in $L^{-\infty, -\infty}(M_1 \times M_2)$, uniformly w.r.t. $\lambda$.
\end{proof}

In order to define  complex powers of an elliptic bisingular operator, we introduce some natural assumptions.
\begin{Ass}
\label{assu}
\begin{enumerate}
\item $A \in S^{m_1, m_2}(M_1, M_2)$ is $\Lambda$-elliptic.
\item $\sigma(A) \cap \Lambda=\emptyset$ (in particular $A$ is invertible).
\item $A$ has homogeneous principal symbols.
\end{enumerate}
\end{Ass}
\begin{remark}
If we consider a $\Lambda$-elliptic operator $A \in L^{m_1, m_2}_{\mathrm{pr}}(M_1 \times M_2)$ with $m_i >0$ ($i=1,2$),
 then $\sigma(A)$ 
is either discrete or the whole of $\C$, because the resolvent is a compact operator (\cite{SH87}, Ch. I). Since 
by Theorem \ref{thmpara} 
we know that for large $\lambda$ the resolvent is well defined, it turns out that the spectrum $\sigma(A)$ 
is discrete. Then, modulo a shift of the operator, we can find a suitable sector such that Assumptions \ref{assu} 
is fulfilled.
\end{remark}

\begin{definition}
Let $A$ be an operator fulfilling Assumptions \ref{assu}. Then, we can define
\begin{equation}
\label{expdef}
A_z:= \frac{i}{2 \pi} \int_{\partial \Lambda^+_\epsilon} \lambda^z (A-\lambda \, \Id)^{-1} d\lambda, \quad \Re (z)<0,
\end{equation}
where $\Lambda_\epsilon=\Lambda \cup \{z \in \C \mid |z|\leq \epsilon\}$.

\noindent The Dunford integral in \eqref{expdef} is convergent because $\|(A-\lambda \, \Id)^{-1}\| = 
O(|\lambda|^{-1})$ for $\lambda$ large enough. As usual, we next define
\[
A^z:= A_{z-k} \circ A^k, \quad \Re(z-k)<0.
\]
\end{definition}

\begin{remark}
\label{reminv}
In Assumptions \ref{assu} we require $\Lambda \cap \sigma (A)= \emptyset$, therefore in particular
 the operator must be invertible.
 It is  possible  to define complex powers of non invertible operator as well, provided the origin is an isolated point 
of the spectrum, see, e.g., \cite{CSS03}. For example, one can define the complex powers of $A=-\Delta \otimes -\Delta$ on
 the torus $\mathbb{S}^1
 \times \mathbb{S}^1$, even if $A$ has an infinite dimensional kernel. 
\end{remark}

\begin{theorem}
\label{expprincipal}
If $A \in L^{m_1, m_2}(M_1, M_2)$ satisfies Assumptions \ref{assu}, then $A^z \in L^{m_1z, m_2z}(M_1 \times M_2)$ and it has homogeneous 
principal symbol. Moreover, by Cauchy Theorem\footnote{In equation \eqref{zsimb} $ a^z_{m_1z, \cdot}, a^z_{\cdot, m_2 z} , a^z_{m_1z, m_2 z}$ 
represent respectively  $\sigma_1^{m_1z}(A^z), \sigma_2^{m_2z}(A^z), \sigma^{m_1z, m_2z}(A^z)$, while $(a_{m_1, \cdot})^z ,(a_{\cdot, m_2})^z $
 are  complex powers of the operators $\sigma^{m_1}_2(A), \sigma^{m_2}_2(A)$ and 
$(a_{m_1, m_2})^z$ is the complex power of the function $\sigma^{m_1, m_2}(A)$.}
\begin{equation}
\label{zsimb}
\begin{split}
&a^z_{m_1z, m_2z}=(a_{m_1, m_2})^z,\\
&a^z_{m_1z, \cdot}=(a_{m_1, \cdot})^z,\\ 
&a^z_{ \cdot, m_2 z }=(a_{\cdot, m_2} )^z.\\
\end{split}
\end{equation} 

\end{theorem}
\begin{proof}
As a consequence of a general version of Fubini's Theorem, denoting by $a^z$ the symbol of $A^z$, we obtain
\[
a^z = \frac{i}{2\pi} \int_{\partial^+\Lambda_\epsilon} \lambda^z (a-\lambda \Id)^{-1}) d\lambda,
 \quad \Re(z)<0.
\]
where $(a- \lambda \Id)^{-1}$ is the symbol of the operator $(A-\lambda \Id)^{-1}$.
By Theorem \ref{thmpara}, we know that $\lambda^2\Big[ (A-\lambda \Id)^{-1}- B(\lambda) \Big] \in L^{-\infty, -\infty}(M_1 \times M_2)$ so,
 up to smoothing symbols, we have
\begin{equation}
\label{expa}
\begin{split}
a^z=& \frac{i}{2\pi} \int_{\partial^+\Lambda_\epsilon} \lambda^z (\tilde{b}(\lambda)) d\lambda\\
 =&\frac{i}{2\pi} \int_{\Omega_{\xi_1, \xi_2}} \lambda^z (\tilde{b}(\lambda)) d\lambda,
\end{split}
\end{equation}
where $\Omega_{\xi_1, \xi_2}$ is as in Lemma \ref{eqell} and the second equality in \eqref{expa} follows by 
Cauchy integral formula.
Now, by Lemma \ref{eqell} and by the explicit form of $\tilde{b}(\lambda)$, 
we get $A^z \in L^{m_1z, m_2z}(M_1 \times M_2)$. 
In order to show that $A^z$ has homogeneous principal symbol, we write 
\[
\begin{split}
 (\tilde{b}(\lambda))=& 
\psi_1 (\sigma^{m_1}(A)- 
\lambda \Id_{M_2})^{-1}+ \psi_2 (\sigma^{m_2}(A)- \lambda \Id_{M_1})^{-1}\\
&- \psi_1 \psi_2 (\sigma^{m_1, m_2}(A)- 
\lambda)^{-1} + c(\lambda),
\end{split}
\] 
where $\lambda c(\lambda) \in S^{-m_1-1,-m_2-1 }(M_1, M_2)$, $\forall \lambda \in \Lambda$.
 We split integral in \eqref{expa} so that
\begin{eqnarray}
\label{parte1}
a^z=&\frac{i}{2\pi} \int_{\partial^+\Lambda_\epsilon} \lambda^z  \psi_1 (\sigma^{m_1}(A)- \lambda \Id_{M_2})^{-1}\\
\label{parte2}
&+\frac{i}{2\pi} \int_{\partial^+\Lambda_\epsilon} \lambda^z  \psi_2 (\sigma^{m_2}(A)  - \lambda \Id_{M_1})^{-1} d\lambda\\
\label{parte3}
&-\frac{i}{2\pi} \int_{\partial^+\Lambda_\epsilon} \lambda^z \psi_1 \psi_2 (\sigma^{m_1, m_2}(A)- \lambda)^{-1} d\lambda\\
\label{parteresto}
&+\frac{i}{2\pi} \int_{\partial^+\Lambda_\epsilon} \lambda^z c(\lambda) d\lambda.
\end{eqnarray}
The theorem follows from theory of complex powers on closed manifolds for the integrals
\eqref{parte1} and \eqref{parte2}, and from Cauchy Theorem for integral
\eqref{parte3}. Finally, we notice that integral \eqref{parteresto} gives a symbol of order $(m_1 z-1,m_2 z-1 )$.
\end{proof}
We now introduce the function $\zeta(A,z)$ of an elliptic operator that satisfies 
Assumptions \ref{assu}. 
 The proof of the following property is similar to the case of compact
 manifolds (see \cite{SH87}, ch. II).  
\begin{proposition}
\label{zetaauto}
Let $A\in L^{m_1, m_2} (M_1 \times M_2)$, $m_i>0$, $i=1,2$,
 be a selfadjoint operator satisfying Assumptions \ref{assu}. Then we have
\[
A^z(u)= \sum_{i \in \N}\lambda_j^z (f_i, u),
\]
where  $\{\lambda_j\}_{j \in \N}$ is the spectrum of $A$, 
and $\{f_j\}_{j \in \N}$ are the corresponding orthonormal eigenfunctions. We define 
\[
\zeta(A,z):= \sum_{j \in \N} \lambda_j^z, \quad \Re (z) < \min \big\{-\frac{n_1}{m_1}, -\frac{n_2}{m_2}\big\}.
\]

\end{proposition}
The definition of $\zeta(A,z)$ in the general case is the following:
\begin{definition}
Let $A\in L^{m_1, m_2}(M_1 \times M_2)$ be an operator satisfying Assumptions \ref{assu}  then
\[
\zeta(A,z):= \int_{M_1 \times M_2} K_{A^z}(x_1,x_2, x_1, x_2)dx_1dx_2, \quad \Re(z)m_1< -n_1, \Re(z)m_2<-n_2,
\]
where $K_{A^z}$ is the kernel of $A^z$. The integral is well defined if 
$\Re(z)m_1< -n_1$ and $ \Re(z)m_2<-n_2$ since, in this case, $A^z$ is trace class.
\end{definition}
\begin{theorem}
\label{exthol}
$K_{A^z}(x_1, x_2,y_1,y_2)$ is a smooth function outside the diagonal. 
Furthermore, $K_{A^z}(x_1, x_2, x_1, x_2)$ restricted to the diagonal 
 can be extended as a meromorphic function on the half plane $\{z \in \C \mid \Re (z)< \min\{- \frac{n_1}{m_1}, 
-\frac{n_2}{m_2}\} + \epsilon \}$ with, at most, poles  at the point $z_{\mathrm{pole}}=\min\{- \frac{n_1}{m_1}, -\frac{n_2}{m_2}\}$. The pole can be 
of order two if $\frac{n_1}{m_1}= \frac{n_2}{m_2}$, otherwise it is a simple pole.
\end{theorem}

\begin{proof}
By definition, the kernel of $A^z$ has the form
\begin{equation}
\label{kernel}
K_{A^z}(x_1, x_2, x_1, x_2)= \frac{1}{(2\pi)^{n_1+ n_2}}\int_{\R^{n_1}}\int_{\R^{n_2}} a^z(x_1, x_2, \xi_1, \xi_2)d\xi_1 d\xi_2.
\end{equation}
First, let us consider the case $\frac{n_1}{m_1}> \frac{n_2}{m_2}$. Then,  if $\Re(z)< 
-\frac{n_1}{m_1}$, $A^z \in L^{m_1 z, m_2 z} (M_1\times M_2)\subseteq L^{-n_1-\epsilon, -n_2-\epsilon}(M_1\times M_2)$; 
hence it is trace class 
and the integral of the kernel is finite. 
We can write $a^z= a^z_{m_1z, \cdot}+ a^z_{r}$, $a^z_{r}\in S^{m_1 z-1, m_2z}(M_1, M_2)$ and we have then
\begin{equation}
\label{ker1}
\begin{split}
K_{A^z}(x,x)=  &\frac{1}{(2\pi)^{n_1+ n_2}}\int_{\R^{n_2}} \int_{|\xi_1| \geq 1}  \big(a^z_{m_1z, \cdot} + a^{z}_{r, \cdot} \big) d\xi_1 d\xi_2  \\
+&\frac{1}{(2\pi)^{n_1+ n_2}} \int_{\R^{n_2}} \int_{|\xi_1|\leq 1}\big(a^z_{m_1z, \cdot} + a^{z}_{r, \cdot} \big)  d\xi_1 d\xi_2.
\end{split}
\end{equation}
The second integral in \eqref{ker1} is an holomorphic function for $\Re(z) \leq -\frac{n_1}{m_1}+ \epsilon$ since we integrate w.r.t. the $\xi_1$
 variable on a compact set. The same conclusion holds for the integral of $a^z_{r, \cdot}$ on the set $\{(\xi_1, \xi_2)\mid |\xi_1|\geq1, \xi_2 \in 
\R^{n_2}\}$ because it has order $(m_1z-1, m_2 z)$. In order to analyze the integral of $a^z_{m_1, \cdot}$, we switch to polar coordinates and we obtain
\begin{equation}
\label{ker2}
\int_{\R^{n_2}}\int_{|\xi_1|\geq 1} a^z_{m_1z, \cdot} d\xi_1 d\xi_2 =-
 \frac{1}{m_1z+n_1} \int_{\R^{n_2}}\int_{\mathbb{S}^{n_1-1}} a_{m_1z, \cdot} d\theta_1 d\xi_2.
\end{equation}
Clearly \eqref{ker2} can be extended as a meromorphic function on $\{z \in \C \mid \Re(z)< -\frac{n_1}{m_1}+ \epsilon\}$, and, moreover, 
by \eqref{zsimb}, we get 
\[
\lim_{z \to -\frac{n_1}{m_1}} \left( z+ \frac{n_1}{m_1}\right) K_{A^z} (x_1, x_2)= - \frac{1}{(2\pi)^{n_1+ n_2} m_1} \int_{\R^{n_2}}\int_{\mathbb{S}^{n_1-1}} 
a_{m_1, \cdot}^{-\frac{n_1}{m_1}} d\theta_1 d\xi_2.
\]
The case $\frac{n_1}{m_1}<\frac{n_2}{m_2}$ is equivalent, by exchanging $m_1$ and $m_2$. 

The case $\frac{n_1}{m_1}=\frac{n_2}{m_2}$ is a bit more delicate, since we have to analyze the 
whole principal symbol. First we write
\begin{equation}
\label{kerdif}
\begin{split}
K_{A^z}(x,x)=\frac{1}{(2\pi)^{n_1+ n_2}} \int_{\R^{n_1}}\int_{\R^{n_2}} &\big(a^z_{m_1z, \cdot} + a^{z}_{\cdot, m_2z}- a^z_{m_1z, m_2z} \big)+ \\
& \big(a^z- a^z_{m_1z, \cdot}- a^z_{\cdot, m_2z}+ a^z_{m_1z, m_2 z} \big) d\xi_1 d\xi_2.
\end{split}
\end{equation}
The definition of principal symbol implies that the second term in \eqref{kerdif} belongs to $S^{m_1z-1, m_2z-1}(M_1, M_2)$, hence the second integral 
is well defined for $\Re(z)<-\frac{n_1}{m_1}+ \epsilon$ and holomorphic for $\Re(z)< -\frac{n_1}{m_1}+ \epsilon$. Now we have 
to analyze the integral of the principal symbol. Splitting $\R^{n_1}\times \R^{n_2}$ into the following four regions
\[
\begin{split}
\{(\xi_1, \xi_2)\mid |\xi_1|<\tau, |\xi_2|<\tau\}, &\quad \{(\xi_2, \xi_2)\mid |\xi_1|\leq \tau, |\xi_2|\geq \tau\},\\
\{(\xi_1, \xi_2)\mid |\xi_1|\geq\tau, |\xi_2|\leq \tau\}, &\quad \{(\xi_2, \xi_2)\mid |\xi_1|>\tau, |\xi_2|> \tau\},
\end{split}
\]
 one gets
\begin{equation}
\label{am1}
\begin{split}
&\int_{\R^{n_1}}\int_{\R^{n_2}}\left( a^z_{m_1z, \cdot}+ a^{z}_{\cdot,m_2z}- a^z_{m_1 z, m_2 z} \right) d\xi_1 d\xi_2=\\
&\frac{\tau^{(m_1 +m_2)z+ n_1+ n_2}}{(m_1z+ n_1)(m_2 z+ n_2)} \int_{\mathbb{S}^{n_1-1}} \int_{\mathbb{S}^{n_2-1}} a^z_{m_1z, m_2z} d\theta_1 d\theta_2\\
&- \frac{\tau^{m_1z + n_1}}{(m_1z+ n_1)}\int_{|\xi_2|\leq \tau}\int_{\mathbb{S}^{n_1-1}}a^z_{m_1z, \cdot}d\theta_1d\xi_2\\
&- \frac{\tau^{m_2z + n_2}}{(m_2z+ n_2)}\int_{|\xi_1|\leq \tau}\int_{\mathbb{S}^{n_2-1}} a^{z}_{ \cdot,m_2z}d\theta_1d\xi_1\\
&- \frac{\tau^{m_1z + n_1}}{(m_1z+ n_1)}\int_{|\xi_2|> \tau}\int_{\mathbb{S}^{n_1-1}} \big( a^{z}_{m_1z,\cdot}- a^z_{m_1z, m_2z}\big) d\theta_1d\xi_1\\
&- \frac{\tau^{m_2z + n_2}}{(m_2z+ n_2)}\int_{|\xi_1|>\tau}\int_{\mathbb{S}^{n_2-1}} \big( a^{z}_{ \cdot,m_2z}- a^z_{m_1z, m_2z}\big) d\theta_1d\xi_1\\
&+ h(z),
\end{split}
\end{equation}
where $h(z)$ is an holomorphic function for $\Re(z)\leq z_{\mathrm{pole}}+\epsilon$. The evaluation of the integrals in \eqref{am1} are similar to Proposition 3.3 in 
 \cite{NI03}, and Theorem 2.2 in \cite{BC10}. This concludes the proof. 
\end{proof}
Since $M_1, M_2$ are closed manifolds, Theorem \ref{exthol} implies the following:
\begin{corollary}
Let $A \in L^{m_1, m_2}(M_1 \times M_2)$ be an operator satisfying Assumptions \ref{assu}; then $\zeta(A, z)$
 is holomorphic for $\Re(z)< \min
\{- \frac{n_1}{m_1}, -\frac{n_2}{m_2}\}$
 and can be extended as a meromorphic function on the half plane $\Re(z) <\min \{- \frac{n_1}{m_1}, -
\frac{n_2}{m_2}\}+ \epsilon$. 
Moreover, the Laurent coefficients of $\zeta(A,z)$ at $z=z_{\mathrm{pole}}= \min \{ - \frac{n_1}{m_1}, -\frac{n_2}{m_2}\}$ are
\begin{equation}
\label{m1>m2}
 \lim_{z\to-\frac{n_1}{m_1}} \left(z+\frac{n_1}{m_1}\right) \zeta(A,z)= -\frac{1}{(2\pi)^{n_1+ n_2} m_1} \iint_{M_1 \times M_2}\int_{\R^{n_2}} 
\int_{\mathbb{S}^{n_1-1}}a_{m_1, \cdot}^{-\frac{n_1}{m_1}}d\theta_1 d\xi_2,
\end{equation}
if  $\frac{n_1}{m_1}>\frac{n_2}{m_2}$ .
\begin{equation}
\label{m2>m1}
  \lim_{z\to-\frac{n_2}{m_2}} \left(z+\frac{n_2}{m_2}\right) \zeta(A,z)= -\frac{1}{(2\pi)^{n_1+ n_2} m_2} \iint_{M_1 \times M_2}\int_{\R^{n_1}} \int
_{\mathbb{S}^{n_2-1}} a_{\cdot, m_2}^{-\frac{n_2}{m_2}}d\theta_2 d\xi_1, 
\end{equation}
if $\frac{n_2}{m_2}>\frac{n_1}{m_1}$.
\begin{equation}
\begin{split}
 \label{m2=m12}
 res^2(A)=&\lim_{z \to -l} (z+l)^2 \zeta(A, z) =\\
 &\frac{1}{(2\pi)^{n_1+ n_2} (m_1 m_2)}\iint_{M_1 \times M_2}\int_{\mathbb{S}^{n_1-1}}
\int_{\mathbb{S}^{n_2-1}}  (a_{m_1, m_2})^{-l} d\theta d\theta',\\
\end{split}
\end{equation}
\begin{equation}
\label{m2=m11}
 \lim_{z\to - l} (z+l) \big(\zeta(A,z)- \frac{res^2(A)}{(z+l)^2} \big)= -TR_{1,2}(A)+ TR_\theta(A),
\end{equation}
where
\begin{equation}
\label{trxy}
\begin{split}
&TR_{1,2}(A):=\\
 &\frac{1}{(2\pi)^{n_1+ n_2}}\lim_{\tau \to \infty}\big(\frac{1}{m_1} \iint_{M_1 \times M_2}\int_{|\xi_2|\leq \tau}\int_{\mathbb{S}^{n_1-1}}
(a_{m_1, \cdot})^{-l}- res^2(A) \log \tau\big)\\
+&\frac{1}{(2\pi)^{n_1+ n_2}}\lim_{\tau \to \infty}\big(\frac{1}{m_2} \iint_{M_1 \times M_2}\int_{|\xi_1|\leq \tau}\int_{\mathbb{S}^{n_2-1}}
(a_{\cdot,m_2})^{-l}- res^2(A) \log \tau\big) 
\end{split}
\end{equation}
and
\begin{equation}
\label{angular}
TR_\theta(A):=\frac{1}{(2\pi)^{n_1+ n_2} (m_1 m_2)}\int_{M_1 \times M_2}\int_{\mathbb{S}^{n_1-1}}\int_{\mathbb{S}^{n_2-1}} a_{m_1, m_2}^{-l}
\log a_{m_1, m_2} d\theta_1 d\theta_2,
\end{equation}
if $\frac{n_1}{m_1}=\frac{n_2}{m_2}=l$.
\end{corollary}
In \eqref{trxy}, $(a_{m_1, \cdot})^l$ and $(a_{\cdot, m_2})^l$ are the symbols of 
the complex powers of the operators
 $a_{m_1, \cdot}(x_1, x_2, \xi_1, D_2)$ and $a_{\cdot, m_2}(x_1, x_2, D_1, \xi_2)$.
In order to obtain the terms in \eqref{m2=m11}, \eqref{trxy}, \eqref{angular}, we notice 
that the constant $\tau$ in \eqref{am1}
 is arbitrary and the Laurent coefficients clearly do not change if we change the partition of $\R^{n_1} \times \R^{n_2}$, therefore we can let $\tau$ tend 
to infinity. In this way both the fourth and fifth integral in \eqref{am1} vanish, due to the 
continuity of the integral w.r.t. 
the domain of integration. 
The evaluation is similar to the proof of Theorem 2.9 in \cite{BC10} and of Proposition 3.3 in \cite{NI03}.

\section{Weyl's formula for bisingular operators}
\label{Weyl}
In this section  we study Weyl's formula for positive selfadjoint bisingular operators that satisfy Assumptions \ref{assu}.
In the sequel we use the following Theorem, proved by J. Aramaki \cite{AR88}:
\begin{theorem}
\label{aramaki}
Let $P$ be a positive selfadjoint operator satisfying Assumptions \ref{assu}. If $\zeta(P,z)$ has the first left pole at the point
 $-z_0$ and\footnote{The Aramaki's Theorem actually requires another assumption on the decay of $\Gamma(z) \zeta(P,z)$
 on vertical strips. 
In this case such condition is fulfilled, in view of the relationship between $\zeta$-function,
 heat trace and gamma function, see \cite{GS96,MSS06b}.}
\[
\zeta(P,z)+ \sum_{j=1}^p \frac{A_j}{(j-1)!} \left(\frac{d}{dz}\right)^{j-1} \frac{1}{z+z_0},
\]
extends to an holomorphic function on the half plane $\{z \in \C \mid \Re(z)< -z_0+ \epsilon\}$, then, setting
\[
N_P(\lambda)= \sum_{t \in \sigma(P), \;t\leq \lambda} 1,
\]
 we have
\[
N_P(\lambda)\sim \sum_{j=1}^p \frac{A_j}{(j-1)!} \left(\frac{d}{ds}\right)^{j-1}\left(\frac{\lambda^s}{s}\right)|_{s=z_0}+ O(\lambda^{z_0-\delta}), \quad \lambda \to \infty,
\]
for a certain $\delta >0$.
\end{theorem}
\begin{theorem}
\label{asint}
Let $A \in L^{m_1, m_2}(M_1 \times M_2)$ be a positive  selfadjoint bisingular 
satisfying Assumptions \ref{assu}, then
\begin{equation}
\label{asbe}
N_A(\lambda)\sim\left\{\begin{array}{ll}
C_1 \lambda^l\log(\lambda)+ C'_1 \lambda^l+ O(\lambda^{l-\delta_1})& \mbox{for } \frac{n_1}{m_1}=\frac{n_2}{m_2}=l\\
 C_2 \lambda^{\frac{n_2}{m_2}}+O(\lambda^{\frac{n_2}{m_2}-\delta_2}) & \mbox{for } \frac{n_2}{m_2}>\frac{n_1}{m_1}\\
 C_3 \lambda^{\frac{n_1}{m_1}}+O(\lambda^{\frac{n_2}{m_2}-\delta_2}) & \mbox{for }\frac{n_2}{m_2}<\frac{n_1}{m_1}
\end{array}\right.
,\quad \lambda \to \infty,
\end{equation}
for certain $\delta_i>0$, $i=1,2,3$. The constants $C_1, C_1', C_2, C_3$  depend only on the principal symbol of $A$.
\end{theorem}
\begin{proof}
We use J. Aramaki's Theorem \ref{asint}, which gives the asymptotic of $N_A(\lambda)$ knowing the first left pole of the zeta function. 
As a simple application we get \eqref{asbe} with

\begin{equation}
\label{risth}
\begin{split}
C_1=& \frac{1}{(2\pi)^{n_1+ n_2} (n_1 \, m_2)} \iint_{M_1 \times M_2}\int_{\mathbb{S}^{n_1-1}}\int_{\mathbb{S}^{n_2-1}} (a_{m_1, m_2})^{-l}d\theta_1 d\theta_2\\
=& \frac{1}{(2\pi)^{n_1+ n_2} (n_2 \, m_1)}\iint_{M_1 \times M_2}\int_{\mathbb{S}^{n_1-1}}\int_{\mathbb{S}^{n_2-1}} (a_{m_1, m_2})^{-l}d\theta_1 d\theta_2;\\
C'_1=& \frac{TR_{1,2}(A)+ TR_{\theta}(A)}{l}-\frac{1}{n_1 n_2}\iint_{M_1 \times M_2}\int_{\mathbb{S}^{n_1-1}}\int_{\mathbb{S}^{n_2-1}} (a_{m_1, m_2})^{-l}d\theta_1 d\theta_2;\\
C_2=&\frac{1}{(2\pi)^{n_1+ n_2} n_2} \iint_{M_1 \times M_2}\int_{\R^{n_1}}\int_{\mathbb{S}^{n_2-1}}(a_{\cdot,m_2})^{-\frac{m_2}{n_2}} d\theta_2 d\xi_1; \\
C_3=& \frac{1}{(2\pi)^{n_1+ n_2} n_1} \iint_{M_1 \times M_2}\int_{\R^{n_2}}\int_{\mathbb{S}^{n_1-1}}(a_{m_1, \cdot})^{-\frac{n_1}{m_1}} d\theta_{1} d\xi_2.
\end{split}
\end{equation}
\end{proof}
\begin{remark}
In this paper we are focused just on bisingular operators with homogeneous principal symbol, since our aim 
is the study of the corresponding  Weyl's formulae.
 We do not introduce classical bisingular operators and we do not investigate the relationship between the poles of the $\zeta$-function and
 Wodzicki Residue defined in \cite{NR06}. Nevertheless, extending the results of section \ref{complexpowers} to classical bisingular operators,
 one can prove that, for a classical elliptic  bisingular operator $A \in L^{m_1, m_2}(M_1 \times M_2)$ that admits complex powers, 
\[
\mathrm{Wres}(A):= m_1 m_2 \lim_{z \to 1} (z-1)^2 \zeta(A,z),
\]
where $\mathrm{Wres}(A)$ is the bisingular Wodzicki residue defined by Nicola and Rodino in \cite{NR06}.
\end{remark}

\section{Examples}
\label{example}
First we consider the operator $A=-\Delta \otimes -\Delta$ on the torus $\mathbb{S}^1 \times \mathbb{S}^1$.
 We clearly have  $\sigma(A)=\{n^2m^2\}_{(n, m) \in \N^2}$. Hence the spectrum is countable and 
consists only of eigenvalues. The eigenvalue $\{0\}$ has an infinite dimensional eigenspace, while all other
 eigen\-spaces have dimension four.
Therefore we get
\begin{equation}
\label{count}
N_{A}(\lambda)= \sum_{0<n^2 \,m^2\leq \lambda} 4.
\end{equation}
Let us define the function $d(h): \N \to \N$ so that $d(h)$ is equal to the number of ways we can write $h=m \cdot n$,
 with $m, n$ natural positive numbers or, equivalently, it is equal to the number of divisors of $h$. This 
function is often called Dirichlet divisor function. By  a simple computation, we obtain
\begin{equation}
\label{dirla}
N_{A}(\lambda^2)=4 \,D(\lambda)= 4 \,\sum_{n\leq \lambda} d(n).
\end{equation}
Noticing that $\zeta(A)= 4 \zeta_R(2z)\zeta_R(2z)$, where $\zeta_R(z)$ is Riemann zeta-function,
 we can easily find the coefficients of the asymptotic expansion and  we have
\begin{equation}
\label{adir}
D(\lambda) \sim \lambda \log (\lambda)+ (2 \gamma-1)\lambda+ O(\lambda^{1-\delta}), \quad \lambda \to \infty,
\end{equation}
where 
\begin{equation}
\label{eumas}
\gamma:= \lim_{\tau \to \infty} \left[\sum_{i=1}^{[\tau]} \frac{1}{i}- \log \tau \right]
\end{equation}
is the well known Euler-Mascheroni constant. 
The asymptotic expansion \eqref{adir} is well known (see \cite{IV03} for an overview on Dirichlet divisor problem;
 see also \cite{IM88,LL94}). 
It is still an open question to understand the behavior of remainder. 
 In \cite{HA16}, G. H. Hardy
proved that $O(\lambda^\frac{1}{4})$ is a lower bound for the third term. The best approximation,
 found by M. Huxley in \cite{HU03}, is $O(\lambda^{c}(\log\lambda)^d)$, where
\[
c:= \frac{131}{416}\sim0,3149038462 \quad d:= \frac{18627}{8320}+1\sim 3,238822115.
\]
The conjecture is that the remainder is $O(\lambda^\frac{1}{4})$.

It is nevertheless interesting to investigate the link between Dirichlet divisor function and the above 
results on the spectral properties of a suitable operators. 
Let us notice that in \eqref{count} we have a slight abuse of notation, since  $N(\lambda)$ was only defined for positive operators. 
In this case $A=-\Delta \otimes -\Delta$ is non-negative, but has a non trivial kernel. In other words we actually consider 
\[
N_{A}:=N_{A \circ (\Id-P_{\ker\,A})}
\]
where $P_{\ker \,A}$ is the projection on the kernel of $A$. This definition is compatible with the definition of complex 
powers of non invertible operators in \cite{CSS03}. The variant of our theory to such a setting, which is possible,  will be not detailed here. Rather, 
let us now consider  the operator $A_c:=(-\Delta + c) \otimes (-\Delta +c)$, $c >0$, 
defined on the torus $\mathbb{S}^1 \times \mathbb{S}^1$.
 Clearly, $A_c$ satisfies Assumptions \ref{assu}; thus we can apply Theorem \ref{asint}. It is easy to see that the 
eigenvalues of $A_c$ are 
 $\{(n^2+ c) (m^2+c)\}_{(n, m) \in \N^2}$, each one with multiplicity four.  Hence
\[
\begin{split}
N(A_c; \lambda^2)&= 4 \;\sharp\{\mbox{ real numbers of the form } (n^2 +c) (m^2+c )\mid\\
 & (n^2 +c) (m^2+c )\leq \lambda, \; n, m \in \N \}= 4 \; D_c(\lambda).
\end{split}
\]
By Theorem \ref{expprincipal}, we know that $\sigma^{-1, -1}(A_c^{-\frac{1}{2}})= (\sigma^{2,2}(A_c))^{-\frac{1}{2}}$ so
the constant $C_1$ in \eqref{risth} can be easily evaluated
\begin{equation}
\label{cost1}
C_1= \frac{1}{2} \frac{1}{(2\pi)^2} (2 \pi)^2 \; 4= 2.
\end{equation}
Since in this case we know the eigenvalue of the operator,  $TR(A_c)$ turns into
\begin{equation}
\label{cost2}
\begin{split}
TR_{1,2}(A_c)&=2 \lim_{\tau \to \infty} \left[\sum_{i=-[\tau]}^{[\tau]} 
\frac{1}{(c+ i^2)^{\frac{1}{2}}} - 2 \log \tau  \right]\\
&= 4 \lim_{\tau \to \infty} \left[\sum_{i=0}^{[\tau]} \frac{1}{(c+ i^2)^{\frac{1}{2}}}- \log \tau \right]= 4 \gamma_c.
\end{split}
\end{equation}
We have named this constant $\gamma_c$ because of the link with the usual constant of Euler-Mascheroni $\gamma$ 
in \eqref{eumas}. 
Notice that,
 letting $c$ tend to $0$, $\gamma_c$ goes to $+ \infty$; while, if $c$ tends to infinity, $\gamma_c$ goes to $-\infty$.
Finally, we obtain
\begin{equation}
\label{asfin}
\begin{split}
D_c(\lambda)&=  \frac{1}{4}N(A_c; \lambda^2)\\
 &\sim \lambda \log (\lambda)+ (2 \gamma_c-1)\lambda+ O(\lambda^{1- \delta}), 
 \quad  \lambda \to \infty.
\end{split}
\end{equation}
In this case, knowing exactly the eigenvalues of the operator, we can check our estimate with 
a numerical experiment. We have checked \eqref{asfin} for $D_c(\lambda)$ 
 with $\lambda=10.000.000$.
In the second column of the Table \ref{tab:1} there is the estimate of the coefficient of first term of the asymptotic
expansion obtained with the software Maple 15, in the third the coefficient  obtained by \eqref{asfin}, and in the fourth the error.
\begin{table}
\caption{1st. term approximation \label{tab:1}}
\begin{tabular}{c||c|c|c|}
c &  1st. term with Maple&  1st. term in  \eqref{asfin}& error \\ \hline
2& 1,024846785 & 1  & 0,024846785 \\ \hline
3 &0,9916281891  & 1 & 0,008371811     \\ \hline
4 & 0,968979304	& 1 & 0,031020696 \\ \hline
5 & 0,951859819	& 1  &0,048140181  \\ \hline
6 & 0,938130598	& 1  & 0,061869402 \\ \hline
7 &0,926687949 &	1 & 0,073312051 \\ \hline
8 & 0,916888721	& 1 & 0,083111279 \\ \hline
9 & 0,908326599	& 1 & 0,091673401 \\ \hline
10 & 0,900728511	& 1  & 0,099271489  \\ \hline
11 & 0,893902326	& 1  & 0,106097674 \\ \hline
12 &0,887707593	& 1  & 0,112292407 \\ \hline
13 &0,882038865	& 1   & 0,117961135  \\ \hline
14 &0,876815128	& 1  & 0,123184872  \\ \hline
15 &0,871972341 &	1 &0,128027659 \\ \hline
16 & 0,867459966 & 	1 & 0,132540034   \\ \hline
17 & 0,863235614	& 1  &0,136764386 \\ \hline
18 &0,859265437 &	1  & 0,140734563 \\ \hline
19 & 0,855520776	& 1 & 0,144479224  \\ \hline
20 & 0,851977951 &	1 &0,148022049
\end{tabular}
\end{table}
We can notice that the error increases with $c$. This is not surprising, since
 \eqref{cost1} does not depend on $c$. In order to make 
 the error smaller, we should increase the number of digits at which we truncate the series $D_c(\lambda)$.
In Table \ref{tab:2} we analyze the coefficient of the second term. 
\begin{table}
\caption{2nd. term approximation \label{tab:2}}
\begin{tabular}{c|| c| c| c|}
c &  2nd. term with Maple&  2nd. term in \eqref{asfin}& error \\ \hline
2& 0,40048285  & 0,401484386 & 0,001001536 \\ \hline
3  & -0,13493765 & -0, 1339381238& 0,000999526 \\ \hline
4 &-0,499994550  & -0,498993281& 0,001001269 \\ \hline
5  &-0,775928050  & -0,774926584 & 0,001001466 \\ \hline
6 &-0,997216950 & -0,996213733 & 0,001003217 \\ \hline
7 &-1,181650650&  -1,180647904 & 0,001002746 \\ \hline
8  &-1,339595550 & -1,3385899520 & 0,001005598 \\ \hline
9  &-1,477600650 &  -1,476592538 & 0,001008112 \\ \hline
10 &-1,600067350 &  -1,599058126 & 0,001009224 \\ \hline
11 &-1,710092450 & -1,7090842470& 0,001008203 \\ \hline
12 &-1,809939750 &-1,808931287 & 0,001008463 \\ \hline
13  &-1,901308850 &-1,9002985710 & 0,001010279  \\ \hline
14 &-1,985505550 & -1,9844949070 & 0,001010643 \\ \hline
15 &-2,063562050 & -2,0625496430 & 0,001012407 \\ \hline
16  &-2,136292950 & -2,1352865400 & 0,001006410 \\ \hline
17 &-2,204381450 &  -2,2033750580 & 0,001006392 \\ \hline
18  &-2,268373150 & -2,2673662890 & 0,001006861 \\ \hline
19 &-2,328729950 & -2,3277195600 & 0,001010390 \\ \hline
20  &-2,385833550 &-2,3848212840 & 0,001012266
\end{tabular}
\end{table}
In this case the error is essentially independent of $c$, this is due to the fact that 
\eqref{cost2} does depend on $c$.

Our spectral approach to Dirichlet Divisor function suggests that others
 \emph{Weyl's formula techniques} (e. g. Fourier Integral 
Operator) could be useful to attack the Dirichlet Divisor conjecture.  

\begin{acknowledgements}

The author wishes to thank Professor L. Rodino for suggesting this topic and Professors
 S. Coriasco, T. Gramchev, F. Nicola and S. Pilipovi\'c for helpful discussions and comments. 
The author wishes to thank also
an anonymous referee for useful suggestions aimed to improve the manuscript.
\end{acknowledgements}

\bibliographystyle{spmpsci}  

\def\cftil#1{\ifmmode\setbox7\hbox{$\accent"5E#1$}\else
  \setbox7\hbox{\accent"5E#1}\penalty 10000\relax\fi\raise 1\ht7
  \hbox{\lower1.15ex\hbox to 1\wd7{\hss\accent"7E\hss}}\penalty 10000
  \hskip-1\wd7\penalty 10000\box7}

\end{document}